\documentclass[11pt,twoside]{article}

\usepackage{amssymb}

\usepackage{latexsym}
\usepackage{graphicx}
\usepackage{amsmath}
\usepackage{hyperref,amsthm}

\numberwithin{equation}{section}

\normalsize \makeatletter

\newcommand{\be}{\begin{equation}}
\newcommand{\ee}{\end{equation}}
\newcommand{\bea}{\begin{eqnarray}}
\newcommand{\eea}{\end{eqnarray}}

\newtheorem{thm}{Theorem}[section]

\newtheorem{lem}[thm]{Lemma}
\newtheorem{prop}[thm]{Proposition}
\newtheorem{defn}[thm]{Definition}

\begin{document}

\title{{\bf Low regularity global well-posedness for the two-dimensional
Zakharov
system}
\thanks{Two authors (D.F. and S.Z.) are supported by
    NSFC 10571158}}

\author{ {\bf Daoyuan Fang}\\ Department of Mathematics,
  Zhejiang University, \\
Hangzhou 310027, China\\
 { dyf@zju.edu.cn}\\
{\bf Hartmut Pecher}\\
Fachbereich Mathematik und Naturwissenschaften\\
Bergische Universit\"at Wuppertal\\
42097 Wuppertal,
Germany\\
 { pecher@math.uni-wuppertal.de}\\
{\bf Sijia Zhong}\\ Department of Mathematics, Southeast
    University, \\
Nanjing 210096, China\\
 { sijiazhong@gmail.com}
}

\maketitle \baselineskip 17pt
\begin{abstract}
The two-dimensional Zakharov system is shown to have a unique global
solution for data without finite energy if the $L^2$-norm of the
Schr\"odinger part is small enough. The proof uses a refined I-method
originally initiated by Colliander, Keel, Staffilani, Takaoka and Tao. A
polynomial growth bound for the solution is also given.

\renewcommand{\thefootnote}{\fnsymbol{footnote}}
\footnotetext{\hspace{-1.8em}{\it AMS 2000 subject classification: Primary:}
35Q55, 35L70  }

\end{abstract}

%
\section{Introduction}
Consider the Zakharov system in space
dimension two: \be\left\{
\begin{array}{ll}
iu_t+\Delta u=nu, \label{1.1}\\
\Box n=n_{tt}-\triangle n=\triangle |u|^2, \\
u(0,x)=u_0(x),\ n(0,x)=n_0(x),\ n_t(0,x)=n_1(x),
\end{array}
\right.\ee where $\triangle$ is the Laplacian in
$\mathbb{R}^2$, $u:[0,T)\times\mathbb{R}^2\rightarrow\mathbb{C}$,
$n:[0,T)\times\mathbb{R}^2\rightarrow\mathbb{R}$.

The Zakharov system was introduced in \cite{Z} to describe the long
wave Langmuir turbulence in a plasma. The function $u$ represents
the slowly varying envelope of the rapidly oscillating electric
field, and the function $n$ denotes the deviation of the ion density
from its mean value. We assume $u_0\in
H^s$, $n_0\in H^l$ and
$n_1\in H^{l-1}$ for some real $s$, $l$.

We consider the Hamiltonian case, that
is, we assume
\be
\label{HC}
\exists v_0 \in L^2(\mathbb{R}^2,\mathbb{R}^2) :
n_1=-\nabla \cdot v_0.
\ee
If $(u,n)$ is a solution of (\ref{1.1}) we have in this case
$$ n_t(t) = -\nabla \cdot(v_0 - \int_0^t \nabla(n+|u|^2)ds) = -\nabla \cdot v(t)
\, , $$
where
$$ v(t) = v_0 - \int_0^t \nabla(n+|u|^2)ds \, , $$
so that
$$ v_t = - \nabla(n+|u|^2) \quad , \quad v(0) = v_0 \, . $$
Thus (\ref{1.1}) can be written in the
form \be\left \{\begin{array}{ll}
iu_t+\Delta u=nu, \label{1.2}\\
n_t=-\nabla\cdot v, \\
v_t=-\nabla n-\nabla|u|^2, \\
u(0,x)=u_0(x),\ n(0,x)=n_0(x),\ v(0,x)=v_0(x).
\end{array}
\right.\ee

This system has two conserved quantities, namely besides mass conservation also
energy conservation (cf. (\ref{MC}) and (\ref{EC}) below).

In one space dimension, the best result with minimal regularity assumptions on
the data was proven by Colliander,
Holmer and Tzirakis \cite{CHT}, who showed global well-posedness in
the case $(s,l) = (0,-1/2)$, the largest $L^2$-based Sobolev space
where local existence is known to hold.

In two space dimensions, Proposition 1.1 of \cite{GTV} tells us that
 the Cauchy problem
(\ref{1.1}) with $(u_0,n_0,n_1)\in H^s\times H^l\times H^{l-1}$ is
locally well posed if $l \ge 0$ and $2s-(l+1)\ge 0$.
Therefore the lowest admissible value of $(s,l)$ is $(\frac{1}{2},0)$.

The main result of our paper now shows that in the Hamiltonian case
these local solutions exist globally, if $ s > 3/4 $ and $ l = 0 $ ,
provided the datum $u_0$ satisfies $ \|u_0\|_{L^2} < \|Q\|_{L^2} $ ,
where $Q$ denotes the ground state of the equation $\Delta Q - Q
+|Q|^2 Q = 0 $ . More precisely we prove
\begin{thm}
\label{thm}
Assume $(u_0,n_0,n_1) \in H^s \times L^2 \times \Lambda^{-1} L^2 $ ,
where $n_1$ fulfills (\ref{HC}), $ 1 > s > 3/4
$ and
 $ \|u_0\|_{L^2} < \|Q\|_{L^2} $. Here $\Lambda$ denotes the operator
$\sqrt{-\Delta}$. Then the system
(\ref{1.1})
has a unique global solution. More precisely, for any $T>0$ there exists a
unique solution
$$(u,n,\Lambda^{-1}n_t)\in X^{s,\frac{1}{2}+}[0,T] \times
\tilde{X}^{0,\frac{1}{2}+}[0,T]
\times \tilde{X}^{0,\frac{1}{2}+}[0,T] $$
 where the spaces $X^{s,b}$ are defined below, and
$\tilde{X}^{0,\frac{1}{2}+}[0,T] :=
X^{0,\frac{1}{2}+}_+[0,T] + X^{0,\frac{1}{2}+}_-[0,T]$. This solution
satisfies
$$(u,n,\Lambda^{-1}n_t) \in C^0([0,T],H^s(\mathbb{R}^2)\times
L^2(\mathbb{R}^2) \times
L^2(\mathbb{R}^2)) $$
and
$$ \|u(t)\|_{H^s} + \|n(t)\|_{L^2} + \|\Lambda^{-1}n_t(t)\|_{L^2} \lesssim
(1+T)^{\frac{1-s}{2s-\frac{3}{2}}+} \, .$$
\end{thm}

Global well-posedness for $s=l+1 \ge 3$ and small data is considered
in \cite{AA1}. Using the Fourier restriction norm method for finite
energy solutions ($s=l+1=1$) Bourgain and Colliander \cite{BC}
proved local well-posedness and also global well-posedness in those
cases where the energy functional controls the $H^1 \times L^2
\times \Lambda^{-1} L^2$ - norm of the solution. This is the case,
if $\|u_0\|_{L^2} < \|Q\|_{L^2}$ .

In order to generalize these results to global existence for data
without finite energy one approach in the last years was initiated
by \cite{CKSTT3}, called the I-method. The main idea is to use a
modified energy functional which is also defined for less regular
functions and not strictly conserved. When one is able to control
its growth in time explicitly, this allows to iterate a modified
local existence theorem to continue the solution to any time $T$ and
moreover to estimate its growth in time. This method was
successfully applied by these authors to several equations which
have a scaling invariance. It was used in \cite{CKSTT3} to improve
Bourgain's global well-posedness results \cite{B1},\cite{B2} for the
(2+1)- and (3+1)-dimensional Schr\"odinger equation. Later it was
applied to the (1+1)-dimensional derivative Schr\"odinger equation
\cite{CKSTT4} and to the KdV and modified KdV equation in
\cite{CKSTT1}.

This method was later more refined by adding a suitable correction term to
the modified
energy functional in \cite{CKSTT},\cite{CKSTT1} and \cite{CKSTT2} in order
to damp out
some oscillations in that functional. It was used in \cite{FZ} and \cite{ZF} to 
prove a $L^2$-concentration result for the Zakharov system and as a corollary 
global well-posedness for small initial data could be improved to 
$\frac{12}{13}<s<1$, $l=0$. This method is also used in our
paper in order to further weaken the regularity assumptions on the data.

It is organized as follows. We transform the system in the usual way into a
first order
system. Then we apply the multiplier I to the Schr\"odinger equation only.
Here for a given $N>>1$, we define smoothing
operators $I_N$: \be\label{2.1}
\widehat{I_Nf}(\xi)=m_N(\xi)\hat{f}(\xi), \ee where
\be\label{2.2}m_N(\xi)=\left\{\begin{array}{ll} 1&|\xi|\leqslant N\\
(\frac{N}{|\xi|})^{1-s}&|\xi|\geqslant2N,
\end{array}\right.\ee and $m_N(\xi)$ is smooth, radial,
nonnegative, and nonincreasing in $|\xi|$. We drop N from the notation
for short when there is no confusion.
We remark that $I:H^s \to H^1$ is a
smoothing operator in the following sense:
$$ \|u\|_{X^{m,b}_{\varphi}} \lesssim \|Iu\|_{X^{m+1-s,b}_{\varphi}}
\lesssim N^{1-s}\|u\|_{X^{m,b}_{\varphi}}.
$$
Here we used the $X^{m,b}_{\varphi}$ - spaces which are defined as follows:
for an
equation of the form
$iu_t-\varphi(-i\nabla_x)u = 0$, where $\varphi$ is a measurable
function, let
$X^{m,b}_{\varphi}$ be
the
completion of ${\cal S}({\bf R}\times{\bf R}^2)$ with respect to
\begin{eqnarray*}
\|f\|_{X^{m,b}_{\varphi}} & := &
\|\langle\xi\rangle^m\langle\tau\rangle^b {\cal
F}(e^{it\varphi(-i\nabla_x)}f(x,t))\|_{L^2_{\xi \tau}} \\ & = &
\|\langle\xi\rangle^m\langle\tau+\varphi(\xi)\rangle^b
\widehat{f}(\xi,\tau)\|_{L^2_{\xi \tau}}.
\end{eqnarray*}
For $\varphi(\xi) = \pm | \xi|$ we use the notation $X^{m,b}_{\pm}$ and for
$\varphi(\xi)
= |\xi|^2$
simply $X^{m,b}$. For a given time interval $I$ we define
$\|f\|_{X^{m,b}(I)} =
\inf_{\tilde{f}_{|I} = f}
\|\tilde{f}\|_{X^{m,b}} $ and similarly $\|f\|_{X^{m,b}_{\pm}(I)} $.

For the modified Zakharov system, where only the Schr\"odinger equation is
multiplied by I, we then prove a local existence
theorem by using the
precise estimates given by \cite{GTV} for the standard Zakharov system in
connection with
an interpolation
type lemma in \cite{CKSTT5}. Our aim is to extract a factor
$\delta^{\kappa}$ with maximal
$\kappa$ from the
nonlinear estimates in order to give an optimal lower bound for the local
existence time
$\delta$ in terms of
the norms of the data.

As it is typical for the $I$-method, one then has to consider in
detail the modified energy functional $H(Iu,n)$ and to control its
growth in time in dependence of the time interval and the parameter
$N$ (cf. the definition of $I$ above). The increment of the energy
has to be small for small time intervals and large $N$. The
increment of $H(Iu,n)$ is not controlled directly but one replaces
$H$ by adding a correction term to it leading to a functional $\tilde{H}$, such
that the difference $H-\tilde{H}$ at a fixed time is small for large
$N$, and, moreover, which is the main technical difficulty, the
growth in time of $\tilde{H}$ can be seen to be small for small time
intervals and large $N$, so that one can control also the growth of
the corresponding norms of the solutions during its time evolution.
This allows to iterate the local existence theorem with time steps
of equal length in order to reach any given fixed time $T$. To
achieve this one has to make the process uniform, which can be
done if $s$ is close enough to 1, namely $ s > 3/4 $.

We use the following notation:
$A\lesssim B$ means there is a universal constant $c>0$, such that
$A\leqslant cB$, and $A\sim B$ when both $A\lesssim B$ and
$B\lesssim A$. $<\xi>=(1+|\xi|^2)^\frac{1}{2}$. $c+$ means $c+\epsilon$, while
$c-$ means $c-\epsilon$, where
$\epsilon>0$ small enough.

There are several properties of the norms $X^{m,b}_{\varphi}$ and
$X^{m,b}_{\varphi}[0,T]$, for which we refer to \cite{GTV}
and \cite{B3}:
\begin{prop}\label{prop2.1}
\begin{enumerate}
\item If $u$ is a
solution of $iu_t+\varphi(-i\nabla_x)u = 0$ with $u(0)=f$ and $\psi$ is a
cutoff
function in
$C^{\infty}_0({\bf R})$ with $supp \, \psi \subset (-2,2)$ , $\psi \equiv
1$ on $[-1,1]$ ,
$ \psi(t) =
\psi(-t) $ , $ \psi(t)\ge 0 $ , $\psi_{\delta}(t):=\psi(\frac{t}{\delta})
\, , $  $
0<\delta \le 1$, we
have
for $b>0$:
$$\|\psi_1 u\|_{X^{m,b}_{\varphi}} \lesssim \|f\|_{H^m}.$$
\item
If $v$ is a solution of the problem $iv_t +\varphi(-i\nabla_x)v =
F $, $ v(0)=0 $, we have for $b'+1 \ge b \ge 0 \ge b' > -1/2$
$$ \|\psi_{\delta} v \|_{X^{m,b}_{\varphi}} \lesssim \delta^{1+b'-b}
\|F\|_{X^{m,b'}_{\varphi}}.$$
\item $u\in
X^{s,b}_\varphi(\mathbb{R}\times\mathbb{R}^2)\Longleftrightarrow
e^{-it\varphi(-i\nabla)}u(t,\cdot)\in H^b(\mathbb{R},
H^s(\mathbb{R}^2))$.

\item For $\frac{2}{q}=1-\frac{2}{r}$ , $2 \le r < \infty$ , the following
Strichartz estimate
holds:
\be\label{2.7}\|u\|_{L^q_tL^r_x}\lesssim\|u\|_{X^{0,\frac{1}{2}+}}.\ee
For the wave part we only use
\be\label{2.8}\|u\|_{L^{\infty}_tL^2_x}\lesssim
\|u\|_{X^{0,\frac{1}{2}+}_ \pm}.\ee

\item For $b>\frac{1}{2}$,
$X_\varphi^{s,b}(\mathbb{R}\times\mathbb{R}^2)\hookrightarrow
C(\mathbb{R}, H^s(\mathbb{R}^2))$, and
$X^{s,b}_{\varphi}[0,T](\mathbb{R}^2)\hookrightarrow
C((-T,T),H^s(\mathbb{R}^2))$.

\item (cf. \cite{P}).  For $0\leqslant b'<b<\frac{1}{2}$, or $0\ge b
> b' > - 1/2$ ,
 $0<T<1$,
\be
\label{Em}
\|u\|_{X^{s,b'}_{\varphi}[0,T]}\lesssim
T^{b-b'}\|u\|_{X^{s,b}_{\varphi}[0,T]} \, .
\ee

\item For $s_1\leqslant s_2$, and $b_1\leqslant b_2$,
$X_\varphi^{s_2,b_2}(\mathbb{R}\times\mathbb{R}^2)\hookrightarrow
X_\varphi^{s_1,b_1}(\mathbb{R}\times\mathbb{R}^2)$.

\item If $f,g\in X^{0,\frac{1}{2}+}$, with
\[{1}_{|\xi_1|\sim N_1}\hat{f}=\hat{f}, {1}_{|\xi_2| \sim
N_2}\hat{g}=\hat{g},\]
 and $N_1\gtrsim N_2$, then
\be\label{2.9}\|fg\|_{L^2_{t,x}}\leqslant
C(\frac{N_2}{N_1})^\frac{1}{2}\|f\|_{X^{0,\frac{1}{2}+}}\|g\|_{X^{0,\frac{1
}{2}+}}.\ee
\end{enumerate}
\end{prop}

Finally, we use the sharp Gagliardo-Nirenberg embedding for $\mathbb{R}^2$,
which could be
found in \cite{W}:
\be \label{GN}
\frac{1}{2} \|u\|_{L^4}^4 \le \frac{\|u\|_{L^2}^2}{\|Q\|_{L^2}^2} \|\nabla
u\|_{L^2}^2 \ee
for $ u \in H^1$ , where $Q$ denotes the ground state for the Schr\"odinger
equation, i.e.
, the unique positive solution (up to translations) of $ \Delta Q - Q +
|Q|^2Q = 0 $.


\section{Local existence}
The system (\ref{1.2}) has the following conserved quantities:
\be
\label{MC} \|u(t)\|_{L^2}  = \|u_0\|_{L^2}, \ee
and
\be
\label{EC} H(u,n,v):= \|\nabla u\|_{L^2}^2 + 1/2(\|n\|_{L^2}^2 + \|v\|_{L^2}^2)
+
\int n|u|^2 \, dx.\ee
This implies an a priori bound for $ \|\nabla
u\|_{L^2} + \|n\|_{L^2} + \|v\|_{L^2}$ by use of the
Gagliardo-Nirenberg embedding (\ref{GN}) under the assumption
$\|u_0\|_{L^2} < \|Q\|_{L^2} $ as follows:
\begin{eqnarray*}
\int n|u|^2 \, dx \le \|n\|_{L^2} \|u\|_{L^4}^2 & \le & \|n\|_{L^2}
\sqrt{2}
\frac{\|u\|_{L^2}}{\|Q\|_{L^2}} \|\nabla u\|_{L^2} \\
& \le & \frac{\epsilon}{2} \|n\|_{L^2}^2 +
\frac{1}{\epsilon}\frac{\|u\|_{L^2}^2}{\|Q\|_{L^2}^2} \|\nabla u\|_{L^2}^2
\, .
\end{eqnarray*}
Choosing $ 1 > \epsilon > \frac{\|u_0\|_{L^2}^2}{\|Q\|_{L^2}^2} $,
we get
\begin{equation}
\label{2.1'} H(u,n,v) \le 2\|\nabla u\|_{L^2}^2 + \|n\|_{L^2}^2 +
\frac{1}{2} \|v\|_{L^2}^2,
\end{equation}
as well as \be (\frac{1}{2}-\frac{\epsilon}{2})\|n\|_{L^2}^2 +
\frac{1}{2} \|v\|_{L^2}^2 + (1-\frac{1}{\epsilon}
\frac{\|u_0\|_{L^2}^2}{\|Q\|_{L^2}^2}) \|\nabla u\|_{L^2}^2 \le
H(u,n,v), \ee so that \be \label{2.3'} \|\nabla u\|_{L^2}^2 +
\|n\|_{L^2}^2 + \|v\|_{L^2}^2 \le c_0 H(u,n,v). \ee
The system
(\ref{1.1}) is transformed into a first order system in $t$ as
follows: with $ n_{\pm}:= n \pm i\Lambda^{-1}n_t $ , i.e. $
n=\frac{1}{2}(n_++n_-)$, $2i\Lambda^{-1}n_t=n_+-n_-$, and
$\overline{n}_+ = n_- $ we get
\begin{eqnarray}
\label{1'}
i u_t + \Delta u & = & \frac{1}{2}(n_+ + n_-)u, \\
\label{2'}
i n_{\pm t} \mp \Lambda n_{\pm} & = & \pm \Lambda(|u|^2), \\
\label{3'} u(0) = u_0 \quad , \quad n_{\pm}(0) & = & n_{\pm 0} \quad
:= \quad n_0 \pm i\Lambda^{-1}n_1.
\end{eqnarray}
One easily checks that the energy $H(u,n,v)$ is transformed into
$$ H(u,n_+) = \|\nabla u\|_{L^2}^2 + \frac{1}{2} \|n_+\|_{L^2}^2 + \frac{1}{2}
\int (n_+ + \overline{n}_+)|u|^2 dx \, , $$
so that (cf. (\ref{2.1'}))
\be
\label{2.1''} H(u,n_+) \lesssim \|\nabla u\|_{L^2}^2 + \|n_+\|_{L^2}^2 \ee
and (cf. (\ref{2.3'}))
\be
\label{2.3''} \|\nabla u\|_{L^2}^2 + \|n_+\|_{L^2}^2 \le c_0 H(u,n_+) \, . \ee

Now, we will apply the I-method (we refer to the introduction for the
definition of I). A crucial role is played by the modified energy
$H(Iu,n_+)$ for the system
\begin{eqnarray}
\label{1''}
i Iu_t + \Delta Iu & = & \frac{1}{2} I[(n_+ +n_-)u], \\
\label{2''}
i n_{\pm t} \mp \Lambda n_{\pm} & = & \pm \Lambda (|u|^2), \\
\label{3''} Iu(0) = Iu_0 \, , \, n_{\pm} (0) & = & n_{\pm 0} = (n_0
\pm i\Lambda^{-1}n_1),
\end{eqnarray}
namely
$$ H(Iu,n_+) := \|\nabla Iu\|_{L^2}^2 + \frac{1}{2} \|n_+\|_{L^2}^2 +
\frac{1}{2} \int (n_+ + \overline{n}_+)|Iu|^2 \, dx. $$ In order to
give a modified local existence result for the system (\ref{1.1}) we
use the following estimates for the nonlinearities by
Ginibre-Tsutsumi-Velo (cf. \cite{GTV}, Lemma 3.4 and 3.5). We denote here and
in the following $X^{s,b}[0,\delta]$ simply by $X^{s,b}$.
\begin{lem}
\label{LemmaGTV}
\begin{itemize}
\item Assume $ 1 > s \ge 0 $ . Then the following estimate holds :
$$ \|n_{\pm} u \|_{X^{s,-\frac{1}{2}+}} \lesssim \delta^{\frac{1}{2}-}
\|n_{\pm}\|_{X^{0,\frac{1}{2}+}_{\pm}} \|u\|_{X^{s,\frac{1}{2}+}} \, . $$
\item Assume $ s \ge 1/2 $ . Then the following estimate holds:
$$ \|\Lambda(|u|^2)\|_{X^{0,-\frac{1}{2}+}_{\pm}} \lesssim
\delta^{\frac{1}{2}-} \|u\|^2_{X^{s,\frac{1}{2}+}} \, . $$
\end{itemize}
\end{lem}

\begin{lem}
\label{Lemma5}
In the case $1 > s \ge 0$  the following estimate holds:
\begin{equation}
\label{*1}
\|I(n_{\pm}u)\|_{X^{1,-\frac{1}{2}+}} \lesssim N^{0+} \delta^{\frac{1}{2}-}
\|n_{\pm} \|_{X^{0,\frac{1}{2}+}_{\pm}} \|Iu\|_{X^{1,\frac{1}{2}+}} \, .
\end{equation}
\end{lem}
\begin{proof} Let $\chi$ be a smooth cutoff function which equals to 1 for
$|\xi| \le N$ , and equals to 0 for $|\xi| \ge 2N$. We estimate as
follows:
$$ \|I(n_{\pm}u)\|_{X^{1,-\frac{1}{2}+}} \le \|I(n_{\pm} {\cal F}^{-1}(\chi
{\cal F}v))\|_{X^{1,-\frac{1}{2}+}} + \|I(n_{\pm} {\cal
F}^{-1}((1-\chi){\cal F} u)\|_{X^{1,-\frac{1}{2}+}} \, $$
where $ v:=Iu $ .

The first term is estimated by Lemma \ref{LemmaGTV} as follows:
\begin{eqnarray}
\nonumber
\|I(n_{\pm}u)\|_{X^{1,-\frac{1}{2}+}}
&\lesssim & N^{0+} \|n_{\pm} {\cal F}^{-1}(\chi {\cal
F}v)\|_{X^{1-,-\frac{1}{2}+}} \lesssim N^{0+} \delta^{\frac{1}{2}-}
\|n_{\pm}\|_{X^{0,\frac{1}{2}+}_{\pm}} \|{\cal F}^{-1}(\chi {\cal
F}v)\|_{X^{1-,\frac{1}{2}+}} \\
\label{*2}
& \lesssim & N^{0+} \delta^{\frac{1}{2}-}
\|n_{\pm}\|_{X^{0,\frac{1}{2}+}_{\pm}} \|Iu\|_{X^{1,\frac{1}{2}+}}\, .
\end{eqnarray}
Next we consider the second term. By Lemma \ref{LemmaGTV} we have
$$ \|n_{\pm} u_1\|_{X^{s,-\frac{1}{2}+}} \lesssim \delta^{\frac{1}{2}-}
\|n_{\pm}\|_{X^{0,\frac{1}{2}+}_{\pm}} \|u_1\|_{X^{s,\frac{1}{2}+}} \, . $$
This means
\begin{equation}
\label{*3}
\left|\int_{\Sigma_3} \frac{ \hat{f}(\xi_1,\tau_1) \hat{u}_1(\xi_2,\tau_2)
\hat{n}_{\pm}(\xi_3,\tau_3) <\xi_1>^s}{<\sigma_1>^{\frac{1}{2}-}
<\sigma_2>^{\frac{1}{2}+} <\sigma>^{\frac{1}{2}+} <\xi_2>^s} d\xi d\tau
\right| \lesssim \delta^{\frac{1}{2}-} \|f\|_{L^2} \|u_1\|_{L^2}
\|n_{\pm}\|_{L^2} \, ,
\end{equation}
where $f\in L^2$, $\Sigma_3$ denotes the set $\xi_1+\xi_2+\xi_3 = 0$
and $\tau_1+\tau_2+\tau_3 = 0$ , $\sigma_j = \tau_j + |\xi_j|^2$
($j=1,\ 2$) and $\sigma = \tau_3 \pm |\xi_3|$ . In order to prove
$$ \|I(n_{\pm} {\cal F}^{-1}((1-\chi){\cal F} u)\|_{X^{1,-\frac{1}{2}+}}
\lesssim  \delta^{\frac{1}{2}-}
\|n_{\pm}\|_{X^{0,\frac{1}{2}+}_{\pm}} \|I {\cal
F}^{-1}((1-\chi){\cal F} u)\|_{X^{1,\frac{1}{2}+}}, $$ we have to
show with $\hat{u}_1:= (1-\chi)\hat{u}$ :
\begin{equation}
\label{*4}
\left|\int_{\Sigma_3 \cap \{|\xi_2|\ge N \}} \frac{m(\xi_1)
\hat{f}(\xi_1,\tau_1)
\hat{u}_1(\xi_2,\tau_2) \hat{n}_{\pm}(\xi_3,\tau_3)
<\xi_1>}{m(\xi_2)<\sigma_1>^{\frac{1}{2}-} <\sigma_2>^{\frac{1}{2}+}
<\sigma>^{\frac{1}{2}+} <\xi_2>} d\xi d\tau \right| \lesssim
\delta^{\frac{1}{2}-} \|f\|_{L^2} \|u_1\|_{L^2} \|n_{\pm}\|_{L^2} \, .
\end{equation}
Because $|\xi_2| \ge N$, we have:\\
If $|\xi_1| \le N$ : $\frac{m(\xi_1)<\xi_1>}{m(\xi_2)<\xi_2>} \sim
(\frac{|\xi_2|}{N})^{1-s}\frac{<\xi_1>}{<\xi_2>} \lesssim
(\frac{<\xi_1>}{N})^{1-s} \frac{<\xi_1>^s}{<\xi_2>^s} \lesssim
\frac{<\xi_1>^s}{<\xi_2>^s}$ . \\
If $|\xi_1| \ge N$ : $\frac{m(\xi_1)<\xi_1>}{m(\xi_2)<\xi_2>} \sim
(\frac{|\xi_2|}{|\xi_1|})^{1-s}\frac{<\xi_1>}{<\xi_2>}  \lesssim
\frac{<\xi_1>^s}{<\xi_2>^s}$.\\ So (\ref{*3}) implies (\ref{*4}).
Thus
\begin{eqnarray*}
\|I(n_{\pm} {\cal F}^{-1}((1-\chi){\cal F} u)\|_{X^{1,-\frac{1}{2}+}}
& \lesssim & \delta^{\frac{1}{2}-} \|n_{\pm}\|_{X^{0,\frac{1}{2}+}_{\pm}} \|I
{\cal F}^{-1}((1-\chi){\cal F} u)\|_{X^{1,\frac{1}{2}+}} \\
& \lesssim &
\delta^{\frac{1}{2}-} \|n_{\pm}\|_{X^{0,\frac{1}{2}+}_{\pm}}
 \|Iu\|_{X^{1,\frac{1}{2}+}} \, .
\end{eqnarray*}
\end{proof}

\begin{prop}
\label{Prop.LE}
Assume $ 1 > s \ge 1/2 $ and $ (u_0,n_{+0},n_{-0}) \in H^s \times L^2
\times L^2 $. Then
there exists
$$ \delta \sim \frac{1}{(\|Iu_0\|_{H^1} + \|n_{+0}\|_{L^2} +
\|n_{-0}\|_{L^2})^{2+}N^{0+}}, $$ such that the system
(\ref{1''}),(\ref{2''}),(\ref{3''}) has a unique local solution in
the time interval $[0,\delta]$ with the property:
$$ \|Iu\|_{X^{1,\frac{1}{2}+}} + \|n_+\|_{X^{0,\frac{1}{2}+}} +
\|n_-\|_{X^{0,\frac{1}{2}+}} \lesssim \|Iu_0\|_{H^1} + \|n_{+0}\|_{L^2} +
\|n_{-0}\|_{L^2} \, . $$
This immediately implies
$$ \|Iu\|_{C^0([0,\delta],H^1)} + \|n_+\|_{C^0([0,\delta],L^2)} +
\|n_-\|_{C^0([0,\delta],L^2)} \lesssim \|Iu_0\|_{H^1} + \|n_{+0}\|_{L^2} +
\|n_{-0}\|_{L^2} \, . $$
\end{prop}
\begin{proof} We use the corresponding integral equations to define a
mapping $S=(S_0,S_1)$ by
\begin{eqnarray*}
S_0(Iu(t)) & = & Ie^{it\Delta}u_0 + \frac{1}{2} \int_0^t e^{i(t-s)\Delta}
I(u(s)(n_+(s)+n_-(s))ds \\
S_1(n_{\pm}(t)) & = & e^{it\Lambda}n_{\pm 0} \pm i \int_0^t e^{\mp
i(t-s)\Delta} \Lambda
(|u(s)|^2) ds \, .
\end{eqnarray*}
Combining Lemma \ref{LemmaGTV} with the interpolation lemma of
\cite{CKSTT5} we get
$$
\| \Lambda(|u|^2)\|_{X^{0,-\frac{1}{2}+}_{\pm}}  \le \| I
\Lambda(|u|^2)\|_{X^{1-s,-\frac{1}{2}+}_{\pm}} \lesssim
\delta^{\frac{1}{2}-}
\|Iu\|^2_{X^{1,\frac{1}{2}+}} \, .$$
This immediately implies
$$ \|S_1(n_{\pm})\|_{X^{0,\frac{1}{2}+}} \lesssim \|n_{\pm 0}\|_{L^2} +
\delta^{\frac{1}{2}-}
\|Iu\|^2_{X^{1,\frac{1}{2}+}} \, .
$$
Using Lemma \ref{Lemma5} we get \be \label{le}
\|S_0(Iu)\|_{X^{1,\frac{1}{2}+}} \lesssim \|Iu_0\|_{H^1} + N^{0+}
\delta^{\frac{1}{2}-} \|n_{\pm}\|_{X_{\pm}^{0,\frac{1}{2}+}}
\|Iu\|_{X^{1,\frac{1}{2}+}}. \ee Choosing $\delta$ as in the
statement of this proposition the standard contraction argument
gives a unique fixed point of $S$, thus the claimed result.
\end{proof}



\section{Estimates for the modified energy}

In this section, let us get the control of the increment of the
modified energy.

As the modified energy is
\[H(Iu,n_+)(t)=\|\nabla
Iu\|_{L^2}^2+\frac{1}{2}\|n_+\|_{L^2}^2+\frac{1}{2}\int_\mathbb{R}(n_++
\bar{n}_+)|Iu|^2d
x,\]
which is not conserved any more, we have to control its growth.

For functions depending on t we drop t from the notation here and in the
following.

First of all, let us define a new quantity
$\tilde{H}(u,n_+)(t)$, which is a slight variant of $H(Iu,n_+)(t)$,
and establish an almost conservation law for that quantity instead.
\begin{defn}\label{def3.1}
Let $k$ be an integer and $\Sigma_k\subset(\mathbb{R}^2)^k$ denote the
space
\[\Sigma_k:=\{(\xi_1,\cdots,\xi_k)\in(\mathbb{R}^2)^k:\xi_1+\cdots+\xi_k=0\},
\]
then
\[\tilde{H}(u,n_+)(t)=-\int_{\Sigma_2}\xi_1m_1\cdot\xi_2m_2\hat{u}(\xi_1)
\hat{\bar{u}}(\xi_2
)+
\frac{1}{2}\int_{\Sigma_2}\hat{n}_+(\xi_1)\hat{\bar{n}}_+
(\xi_2)+\frac{1}{2}\int_{\Sigma
_3}
\sigma\hat{u}(\xi_1)\hat{\bar{u}}(\xi_2)(\hat{n}_++\hat{\bar{n}}_+)(\xi_3)
\]
is called the refined energy, where $m_i=m_N(\xi_i)$, and
$\sigma=\frac{|\xi_1|^2m_1^2-|\xi_2|^2m_2^2}{|\xi_1|^2-|\xi_2|^2}$.
\end{defn}

Then we shall show the following:
\begin{prop}\label{prop.3.1}(Fixed-time difference)
For $s>\frac{1}{2}$ we have
\be\label{3.1}|H(Iu,n_+)(t)-\tilde{H}(u,n_+)(t)|\lesssim
N^{-1+}\|Iu(t)\|_{H^1_x(\mathbb{R}^2)}^2\|n_+(t)\|_{L^2_x(\mathbb{R}^2
)}.\ee
\end{prop}

\begin{prop}\label{prop.3.2}
(Almost conservation law) For $s>\frac{1}{2}$, if
$(Iu,n_+,n_-)$ is the solution to the Cauchy problem
(\ref{1''}),(\ref{2''}),(\ref{3''})
on the time interval $[0,\delta]$ with initial data $(Iu_0,n_{+0},n_{-0})\in
H^1(\mathbb{R}^2)\times L^2(\mathbb{R}^2)\times L^2(\mathbb{R}^2)$, then we have
 \bea
&&|\tilde{H}(u,n_+)(\delta)-\tilde{H}(u,n_+)(0)|\nonumber\\
&\lesssim&
N^{-\frac{1}{2}+}\delta^{\frac{1}{2}-}\|Iu\|_{X^{1,\frac{1}{2}+}}^2\|n_+\|_
{X^{0,\frac{1}{2}+}_
{+}}+(N^{-2+}+N^{-1+}\delta^{\frac{1}{2}-})\|Iu\|_{X^{1,\frac{1}{2}+}}^2\|n
_+\|_{X^{0
,\frac{1}{2}+}_+}^2.
\label{3.2}\eea
\end{prop}
The remaining part of this section is devoted to prove the above two
propositions.\\
{\bf Proof of Proposition \ref{prop.3.1}:}
Since
\[H(Iu,n_+)=-\int_{\Sigma_2}\xi_1m_1\cdot\xi_2m_2\hat{u}(\xi_1)\hat{\bar{u}}
(\xi_2)+\frac{1}{2}\int_{\Sigma_2}\hat{n}_+(\xi_1)\hat{\bar{n}}_+(\xi_2)+
\frac{1}{2}\int_{\Sigma_3}m_1m_2\hat{u}(\xi_1)\hat{\bar{u}}(\xi_2)
(\hat{n}_++\hat{\bar{n}}_+)(\xi_3),\] and \[
\tilde{H}(u,n_+)=-\int_{\Sigma_2}\xi_1m_1\cdot\xi_2m_2\hat{u}(\xi_1)
\hat{\bar{u}}(\xi_2)+\frac{1}{2}\int_{\Sigma_2}\hat{n}_+(\xi_1)
\hat{\bar{n}}_+(\xi_2)+\frac{1}{2}\int_{\Sigma_3}\sigma\hat{u}(\xi_1)
\hat{\bar{u}}(\xi_2)(\hat{n}_++\hat{\bar{n}}_+)(\xi_3),\] we have
\be\label{1}H(Iu,n_+)-\tilde{H}(u,n_+)=\frac{1}{2}\int_{\Sigma_3}(m_1m_2-
\sigma)\hat{u}(\xi_1)\hat{\bar{u}}(\xi_2)(\hat{n}_++\hat{\bar{n}}_+)(\xi_3).
\ee

We use a dyadic decomposition with $N_i\leqslant|\xi_i|\leqslant 2N_i$. As
the complex conjugates will play no role here, we can suppose
$N_1\geqslant N_2$.

If $N_2\leqslant N_1<<N$, then by the
definition of $m_N$ and $\sigma$, the integral vanishes. Hence, we suppose
$N_1\gtrsim N$.

Therefore, it remains to prove under these assumptions
\be\label{2}
I=|\int_{\Sigma_3}(m_1m_2-\sigma)\hat{u}(\xi_1)\hat{\bar{u}}(\xi_2)(\hat{n}
_++\hat{\bar{n}}_+)(\xi_3)|\lesssim
N^{-1+}N_1^{0-}\|Iu\|_{H^1}^2\|n_+\|_{L^2}.\ee

\begin{lem}\label{lem3.1}
Under the above assumption, $\sigma$ is bounded.
\end{lem}
\begin{proof}
Case 1. $N_2<<N\lesssim N_1$.
\be\label{3.5}|\sigma|=|\frac{|\xi_1|^2m_1^2-|\xi_2|^2}{|\xi_1|^2-|\xi_2|^2
}|\sim\frac{||\xi_1|^2m_1^2-|\xi_2|^2|}{|\xi_1|^2}\lesssim
m_1^2+\frac{N_2^2}{N_1^2}\lesssim1.\ee

Case 2. $N\lesssim N_2\leqslant N_1$.

\be\label{3.6}|\sigma|=|\frac{|\xi_1|^2m_1^2-|\xi_2|^2m_2^2}{|\xi_1|^2-|\xi
_2|^2}|=|\frac{f(|\xi_1|)-f(|\xi_2|)}{|\xi_1|^2-|\xi_2|^2}|,\ee
where $f(r)=r^2m_N(r)^2$ and $r\gtrsim N$. Thus $|f'(r)|\lesssim
N^{2(1-s)}r^{2s-1}$, which is monotone increasing w.r.t. $r\gtrsim
N$, because $s>\frac{1}{2}$. Hence,
\be\label{3.7}|\sigma|\lesssim|\frac{N^{2(1-s)}|\xi_1|^{2s-1}(|\xi_1|-|\xi_
2|)}{(|\xi_1|+|\xi_2|)(|\xi_1|-|\xi_2|)}|\lesssim
\frac{N^{2(1-s)}|\xi_1|^{2s-1}}{|\xi_1|}=(\frac{N}{|\xi_1|})^{2(1-s)}
\lesssim1.\ee
\end{proof}
This lemma implies $|m_1m_2-\sigma|\lesssim1$, and for $s > 1/2$ we
get \bea
I&\lesssim&\frac{1}{m(N_1)m(N_2)}\|Iu_1\|_{L_x^{2+}}\|Iu_2\|_{L_x^{\infty-}}\|n_
+
\|_{L_x^2}\nonumber\\
&\lesssim&(\frac{N_1}{N})^{1-s}((\frac{N_2}{N})^{1-s}+1)\frac{1}{N_1^{1-}}
\|Iu\|_{H^1}^2\|n_+\|_{L^2}\nonumber\\
&\lesssim&N^{-1+}N_1^{0-}\|Iu\|_{H^1}^2\|n_+\|_{L^2}\label{3}.\eea
Here and in the following we abuse notation and denote
$$ m(N_i) = \inf_{|\xi_i| \sim N_i} m(\xi_i) \sim \sup_{|\xi_i| \sim N_i}
m(\xi_i) \, .$$
This completes the proof of Proposition \ref{prop.3.1}.\\[1cm]
{\bf Proof of Proposition \ref{prop.3.2}:}
By system (\ref{1''}),(\ref{2''})
\bea &&\frac{d}{dt}\tilde{H}(u,n_+)(t)\nonumber\\
&=&-\int_{\sum_2}\xi_1m_1\cdot\xi_2m_2\hat{u}_t(\xi_1)\hat{\bar{u}}(\xi_2)-
\int_{\sum_2}\xi_1m_1\cdot\xi_2m_2\hat{u}(\xi_1)\hat{\bar{u}}_t(\xi_2)
\nonumber\\
&&+\frac{1}{2}\int_{\sum_2}\hat{n}_{+t}(\xi_1)\hat{\bar{n}}_{+}(\xi_2)+
\frac{1}{2}\int_{\sum_2}\hat{n}_{+}(\xi_1)\hat{\bar{n}}_{+t}(\xi_2)
\nonumber\\
&&+\frac{1}{2}\int_{\sum_3}\sigma\hat{u}_t(\xi_1)\hat{\bar{u}}(\xi_2)
(\hat{n}_++\hat{\bar{n}}_+)(\xi_3)+\frac{1}{2}\int_{\sum_3}\sigma\hat{u}
(\xi_1) \hat{\bar{u}}_t(\xi_2)(\hat{n}_++\hat{\bar{n}}_+)(\xi_3)\nonumber\\
&&+\frac{1}{2}\int_{\sum_3}\sigma\hat{u}(\xi_1)\hat{\bar{u}}(\xi_2)(\hat{n}
_{+t}+\hat{\bar{n}}_{+t})(\xi_3)\nonumber\\
&=&-\frac{i}{2}\int_{\sum_3}(1-\sigma)|\xi_3|\hat{u}(\xi_1)\hat{\bar{u}}
(\xi_2)\hat{\bar{n}}_+(\xi_3)
+\frac{i}{2}\int_{\sum_3}(1-\sigma)|\xi_3|\hat{u}(\xi_1)\hat{\bar{u}}
(\xi_2)\hat{n}_+(\xi_3)\nonumber\\
&&+2Im\int_{\sum_4}(\frac{|\xi_{23}|^2m_{23}^2-|\xi_2|^2m_2^2}{|\xi_{23}|^2
-|\xi_2|^2}-m_{23}^2)\hat{u}(\xi_1)\hat{\bar{u}}(\xi_2)\hat{n}(\xi_3)
\hat{n}(\xi_4)\nonumber\\
&=&-\frac{i}{2}\int_{\sum_3}(1-\sigma)|\xi_3|\hat{u}(\xi_1)\hat{\bar{u}}(
\xi_2)\hat{\bar{n}}_+(\xi_3)
+\frac{i}{2}\int_{\sum_3}(1-\sigma)|\xi_3|\hat{u}(\xi_1)\hat{\bar{u}}
(\xi_2)\hat{n}_+(\xi_3)\nonumber\\
&&+2Im\int_{\sum_4}\frac{|\xi_2|^2(m_{23}^2-m_2^2)}{|\xi_{23}|^2-|\xi_2|^2}
\hat{u}(\xi_1)\hat{\bar{u}}(\xi_2)\hat{n}(\xi_3)\hat{n}(\xi_4)\label{3.12},
\eea where $\xi_{ij}=\xi_i+\xi_j$, $m_{ij}=m_N(\xi_i+\xi_j)$, and
the expression for $\sigma$ is given above.

Integrating with respect to $t$ on $[0,\delta]$, we have
\bea&&\tilde{H}(u,n_+)(\delta)-\tilde{H}(u,n_+)(0)\nonumber\\
&=&-\frac{i}{2}\int_0^\delta\int_{\sum_3}(1-\sigma)|\xi_3|
\hat{u}(\xi_1)\hat{\bar{u}}(\xi_2)\hat{\bar{n}}_+(\xi_3)
+\frac{i}{2}\int_0^\delta\int_{\sum_3}(1-\sigma)|\xi_3|\hat{u}(\xi_1)
\hat{\bar{u}}(\xi_2)\hat{n}_+(\xi_3)\nonumber\\
&&+2Im\int_0^\delta\int_{\sum_4}\frac{|\xi_2|^2(m_{23}^2-m_2^2)}{|\xi_{23}|
^2-|\xi_2|^2}\hat{u}(\xi_1)\hat{\bar{u}}(\xi_2)\hat{n}(\xi_3)\hat{n}(\xi_4)
\label{3.13}.
\eea

Because the complex
conjugates play no role here, there are two kinds of terms
we have to deal with:

\be\label{3.14} II =
|\int_0^\delta\int_{\sum_3}(1-\sigma)|\xi_3|\hat{u}(\xi_1)
\hat{\bar{u}}(\xi_2)\hat{n}_+(\xi_3)|,\ee and \be\label{3.15} III =
|\int_0^\delta\int_{\sum_4}
\frac{|\xi_2|^2(m_{23}^2-m_2^2)}{|\xi_{23}|^2-|\xi_2|^2}\hat{u}(\xi_1)
\hat{\bar{u}}(\xi_2)\hat{n}_+(\xi_3) \hat{n}_+(\xi_4)|.\ee

First we prove
\be\ II \lesssim
N^{-\frac{1}{2}+}N_1^{0-}\delta^{\frac{1}{2}-}\|Iu\|_{X^{1,\frac{1}{2}+}}^2
\|n_+\|_{X^{0,\frac{1}{2}+}_+},\ee
where we can assume as above $N_2\leqslant N_1$, $N_1\gtrsim N$, and
$N_3\lesssim N_1$.

As $|1-\sigma|\lesssim 1$, \bea
II&\lesssim&N_3\frac{1}{m(N_1)m(N_2)}\|Iu_1Iu_2\|_{L^2_{t,x}}\|n_+\|_{L^2_{t,x}}
\nonumber\\
&\lesssim&N_3(\frac{N_1}{N})^{1-s}((\frac{N_2}{N})^{1-s}+1)(\frac{N_2}{N_1}
)^{\frac{1}{2}}\|Iu_1\|_{X^{0,\frac{1}{2}+}}\|Iu_2\|_{X^{0,\frac{1}{2}+}}\|
n_+\|_{X^{0,0}_+}\nonumber\\
&\lesssim&N_3(\frac{N_1}{N})^{1-s}((\frac{N_2}{N})^{1-s}+1)(\frac{N_2}{N_1}
)^{\frac{1}{2}}\frac{1}{N_1}\frac{1}{<N_2>}\delta^{\frac{1}{2}-}\|Iu\|_{X^{
1,\frac{1}{2}+}}^2\|n_+\|_{X^{0,\frac{1}{2}+}_+}\nonumber\\
&\lesssim&N^{-\frac{1}{2}-}N_1^{0-}\delta^{\frac{1}{2}-}\|Iu\|_{X^{1,\frac{
1}{2}+}}^2\|n_+\|_{X^{0,\frac{1}{2}+}_+}.\eea

Next we prove
\be\label{5}III
\lesssim(N^{-2+}+N^{-1+}\delta^{\frac{1}{2}+})\|Iu\|_{X^{1,\frac{1}{2}+}}^2
\|n_+\|_{X^{0,\frac{1}{2}+}}^2.\ee

If both $N_2$ and $N_3<<N$, then $m_{23}^2-m_2^2=0$, which is trivial.
Thus we suppose $N_2$ or $N_3\gtrsim N$.

Case 1. $N_2<<N_3$, and $N_3\gtrsim N$.

\be\label{6}|\frac{|\xi_2|^2(m_{23}^2-m_2^2)}{|\xi_{23}|^2-|\xi_2|^2}|
\lesssim\frac{|\xi_2|^2}{||\xi_2+\xi_3|^2-|\xi_2|^2|}\lesssim\frac{|\xi_2|^
2}{| \xi_3|^2}.\ee

Since $N_2<<N_3$ and $\xi_1+\xi_2+\xi_3+\xi_4=0$,
$N_4\lesssim\max\{N_1,\ N_3\}$, and $N_{max}\lesssim\max\{N_1,\
N_3\}$.

Subcase 1.1. $N_1<<N$.

So $N_{max}\sim N_3$.

\bea
III&\lesssim&\frac{N_2^2}{N_3^2}\frac{1}{m(N_1)m(N_2)}\|Iu_1\|_{L^{\infty-}_{t,x
}
}\|Iu_2\|_{L^{2+}_tL^{\infty-}_x}\|n_{+3}\|_{L^2_tL^{2+}_x}
\|n_{+4}\|_{L^\infty_tL^2_x}\nonumber\\
&\lesssim&\frac{N_2^2}{N_3^2}((\frac{N_2}{N})^{1-s}+1)\|Iu_1\|_{L^{\infty-}
_tH^1_x}\|Iu_2\|_{X^{0,\frac{1}{2}+}}\|n_{+3}\|_{L^2_tH^{0+}_x}\|n_{+4}\|_{
X^{0,\frac{1}{2}+}_+}\nonumber\\
&\lesssim&\frac{N_2^2}{N_3^2}((\frac{N_2}{N})^{1-s}+1)\frac{1}{<N_2>}N_3^{0
+}\delta^{\frac{1}{2}-}\|Iu\|_{X^{1,\frac{1}{2}+}}^2\|n_+\|_{X^{0,\frac{1}{
2}+}_+}^2\nonumber\\
&\lesssim&N^{-1+}\delta^{\frac{1}{2}-}N_{max}^{0-}\|Iu\|_{X^{1,\frac{1}{2}+
}}^2\|n_+\|_{X^{0,\frac{1}{2}+}_+}^2\label{7}.
\eea

Subcase 1.2. $N_1\gtrsim N$.

\bea
III&\lesssim&\frac{N_2^2}{N_3^2}\frac{1}{m(N_1)m(N_2)}\|Iu_1\|_{L^{2+}_t
L^{\infty
-}_x}\|Iu_2\|_{L^{2+}_tL^{\infty-}_x}\|n_{+3}\|_{L^{\infty-}_tL^{2+}_x}\|n_
{+4}\|_{L^\infty_tL^2_x}\nonumber\\
&\lesssim&\frac{N_2^2}{N_3^2}(\frac{N_1}{N})^{1-s}((\frac{N_2}{N})^{1-s}+1)
\|Iu_1\|_{X^{0,\frac{1}{2}+}}\|Iu_2\|_{X^{0,\frac{1}{2}+}}
\|n_{+3}\|_{L^{\infty-}_tH^{0+}_x}\|n_{+4}\|_{X^{0,\frac{1}{2}+}_+}
\nonumber\\
&\lesssim&\frac{N_2^2}{N_3^2}(\frac{N_1}{N})^{1-s}((\frac{N_2}{N})^{1-s}+1)
\frac{1}{N_1}\frac{1}{<N_2>}N_3^{0+}\|Iu\|_{X^{1,\frac{1}{2}+}}^2\|n_+\|_{X
^{0,\frac{1}{2}+}_+}^2\nonumber\\
&\lesssim&N^{-2+}N_{max}^{0-}\|Iu\|_{X^{1,\frac{1}{2}+}}^2\|n_+\|_{X^{0,
\frac{1}{2}+}_+}^2.\label{8}
\eea

Case 2. $N_3<<N_2$, $N_2\gtrsim N$.

So \be\label{9}
|\frac{|\xi_2|^2(m_{23}^2-m_2^2)}{|\xi_{23}|^2-|\xi_2|^2}|=|\frac{|\xi_2|^2
(m_N(\xi_2+\xi_3)^2-m_N(\xi_2)^2)}{(|\xi_{23}|+|\xi_2|)||\xi_{23}|
-|\xi_2||)}|\lesssim\frac{|\xi_2|^2
N^{2(1-s)}|\xi_2|^{2s-3}||\xi_{23}|-|\xi_2||}{|\xi_2| ||\xi_{23}| -
|\xi_2||} \lesssim m_2^2,
\ee
and $N_4\lesssim\max\{N_1,\ N_2\}$, $N_{max}\lesssim\max\{N_1,\
N_2\}$.

Subcase 2.1. $N_1<<N$.

So $N_{max}\sim N_2$.

As above
\bea
III&\lesssim&m(N_2)\frac{1}{m(N_1)m(N_2)}\|Iu_1\|_{L^{\infty-}_{t,x}}
\|Iu_2\|_{L^{2+
}_tL^{\infty-}_x}\|n_{+3}\|_{L^2_tL^{2+}_x}\|n_{+4}\|_{L^\infty_tL^2_x}
\nonumber\\
&\lesssim&\frac{1}{N_2}N_3^{0+}\delta^{\frac{1}{2}-}
\|Iu\|_{X^{1,\frac{1}{2}+}}^2\|n_+\|_{X^{0,\frac{1}{2}+}_+}^2\nonumber\\
&\lesssim&N^{-1+}\delta^{\frac{1}{2}-}N_{max}^{0-}\|Iu\|_{X^{1,\frac{1}{2}+
}}^2\|n_+\|_{X^{0,\frac{1}{2}+}_+}^2\label{10}.\eea

Subcase 2.2. $N_1\gtrsim N$.

As in subcase 1.2, \bea
III&\lesssim&m(N_2)\frac{1}{m(N_1)m(N_2)}\|Iu_1\|_{L^{2+}_tL^{\infty-}_x}
\|Iu_2\|_{L
^{2+}_tL^{\infty-}_x}\|n_{+3}\|_{L^{\infty-}_tL^{2+}_x}\|n_{+4}\|_{L^\infty
_tL^2_x}\nonumber\\
&\lesssim&(\frac{N_1}{N})^{1-s}\frac{1}{N_1}\frac{1}{N_2}N_3^{0+}\|Iu\|_{X^
{1,\frac{1}{2}+}}^2\|n_+\|_{X^{0,\frac{1}{2}+}_+}^2\nonumber\\
&\lesssim&N^{-(1-s)}N_2^{-1+}N_1^{-s}\|Iu\|_{X^{1,\frac{1}{2}+}}^2\|n_+\|_{
X^{0,\frac{1}{2}+}_+}^2\nonumber\\
&\lesssim&N^{-2+}N_{max}^{0-}\|Iu\|_{X^{1,\frac{1}{2}+}}^2
\|n_+\|_{X^{0,\frac{1}{2}+}_+}^2.\label{11}
\eea

Case 3 $N_2\sim N_3\gtrsim N$.

Hence $N_4\lesssim\max\{N_1,\ N_2,\ N_3\}\sim\max\{N_1,\ N_2\}$, and
$N_{max}\lesssim\max\{N_1,\ N_2\}$.

We have
$$|\frac{|\xi_2|^2(m_{23}^2-m_2^2)}{|\xi_{23}|^2-|\xi_2|^2}|=|\frac{|\xi_{23
}|^2m_{23}^2-|\xi_2|^2 m_2^2}{|\xi_{23}|^2-|\xi_2|^2}
-m_{23}^2|\lesssim|\frac{|\xi_{23}|^2 m_{23}^2-|\xi_2|^2
m_2^2}{|\xi_{23}|^2-|\xi_2|^2}|+m_{23}^2 \, .$$
By the proof of Lemma \ref{lem3.1}, the above expression is
bounded.

Subcase 3.1. $ N_1 \gtrsim N$.

\bea III&\lesssim&\frac{1}{m(N_1)m(N_2)}\|Iu_1\|_{L^{2+}_tL^{\infty-}_x}
\|Iu_2\|_{L^{2+}_tL^{\infty-}_x}\|n_{+3}\|_{L^{\infty-}_tL^{2+}_x}\|n_{+4}
\|_{L^\infty_tL^2_x}\nonumber\\
&\lesssim&(\frac{N_1}{N})^{1-s}(\frac{N_2}{N})^{1-s}\frac{1}{N_1}
\frac{1}{N_2}N_3^{0+}\|Iu\|_{X^{1,\frac{1}{2}+}}^2\|n_+
\|_{X^{0,\frac{1}{2}+}_+}^2 \nonumber\\
&\lesssim&N^{-2+}N_{max}^{0-}\|Iu\|_{X^{1,\frac{1}{2}+}}^2\|n_+
\|_{X^{0,\frac{1}{2}+}_+}^2.\label{12}\eea

Subcase 3.2. $ N_1 \lesssim N$.

Thus $ N_{max} \lesssim N_2$.
\bea
III & \lesssim & \frac{1}{m(N_1) m(N_2)} \|Iu_1\|_{L_{t,x}^{\infty -}}
\|n_{+3}\|_{{L_t^{2+}} L_x^{2+}} \|Iu_2\|_{L^{2+}_t L^{\infty}_x}
\|n_{+4}\|_{L_t^{\infty} L_x^2} \nonumber \\
& \lesssim & (\frac{N_2}{N})^{1-s} \|Iu_1\|_{X^{1,\frac{1}{2}+}}
\|n_{+3}\|_{X_+^{0+,0+}} \|Iu_2\|_{X^{0+,\frac{1}{2}+}}
\|n_{+4}\|_{X^{0,\frac{1}{2}+}_+} \nonumber \\
& \lesssim & (\frac{N_2}{N})^{1-s} N_3^{0+} N_2^{-1+} \delta^{\frac{1}{2}-}
\|Iu\|_{X^{1,\frac{1}{2}+}}^2 \|n_+\|_{X^{0,\frac{1}{2}+}_+}^2 \nonumber \\
& \lesssim & N^{-1+} N_{max}^{0-} \delta^{\frac{1}{2}-}
\|Iu\|_{X^{1,\frac{1}{2}+}}^2 \|n_+\|_{X^{0,\frac{1}{2}+}_+}^2.
\nonumber \eea


\section{Proof of Theorem \ref{thm}}

\begin{proof}
The data satisfy the estimate
$$ \|Iu_0\|_{H^1}  \le  c N^{1-s}\|u_0\|_{H^s} \, . $$
We use our local existence theorem on $[0,\delta]$, where
$$ \delta \sim \frac{1}{(\|Iu_0\|_{H^1} + \|n_{+0}\|_{L^2} +
\|n_{-0}\|_{L^2})^{2+}N^{0+}}. $$ and conclude
\begin{eqnarray}
\nonumber
\lefteqn{ \|Iu\|_{X^{1,\frac{1}{2}+}[0,\delta]} +
\|n_+\|_{X^{0,\frac{1}{2}+}_+[0,\delta]} +
\|n_-\|_{X^{0,\frac{1}{2}+}_-[0,\delta]} } \\
&& \le  c(\|Iu_0\|_{H^1} + \|n_{+0}\|_{L^2} + \|n_{-0}\|_{L^2}) \le c_2
N^{1-s} \label{**}.
\end{eqnarray}
From (\ref{2.1''}) we get
$$ H(Iu_0,n_{+0}) \le c_0 (\|Iu_0\|_{H^1}^2 + \|n_{+0}\|_{L^2}^2) \le
\overline{c} N^{2(1-s)}, $$ and from (\ref{2.3''})
$$ \|\Lambda Iu_0\|_{L^2}^2 + \|n_{+0}\|_{L^2}^2 + \|n_{-0}\|_{L^2}^2 \le
\widehat{c} N^{2(1-s)} \quad , \quad \|Iu_0\|_{L^2} \le
\|u_0\|_{L^2} =: M $$ with $ \widehat{c} = \widehat{c}(\overline{c})
$. Thus the constant in (\ref{**}) depends only on $\overline{c}$
and $M$, i.e. $c_2 = c_2(\overline{c},M)$.

In order to reapply the local existence result with time intervals
of equal length we need a uniform bound of the solution at time
$t=\delta$ and $t=2\delta$ etc. which follows from a uniform control
over the energy by (\ref{2.3''}). The increment of the energy is
controlled by Proposition \ref{prop.3.1} and Proposition \ref{prop.3.2} as
follows:
\begin{eqnarray*}
\lefteqn{ |H(Iu(\delta),n_+(\delta)) - H(Iu_0,n_{+0})| } \\
& \le & |H(Iu(\delta),n_+(\delta)) - \tilde{H}(u(\delta),n_+(\delta))| \\
& &  +|\tilde{H}(u(\delta),n_+(\delta)) - \tilde{H}(u_0,n_{+0})| + |
\tilde{H}(u_0,n_{+0}) - H(Iu_0,n_{+0})| \\
& \le & c [N^{-1+} \|Iu(\delta)\|_{H^1}^2 \|n_+(\delta)\|_{L^2} +
N^{-\frac{1}{2}+} \delta^{\frac{1}{2}+}
\|n_+\|_{X^{0,\frac{1}{2}+}_+[0,\delta]}
\|Iu\|_{X^{1,\frac{1}{2}}[0,\delta]}^2 \\
&& + (N^{-2+} + N^{-1+} \delta^{\frac{1}{2}-})
\|n_+\|_{X^{0,\frac{1}{2}+}_+[0,\delta]}^2
\|Iu\|_{X^{1,\frac{1}{2}}[0,\delta]}^2
 + N^{-1+} \|Iu_0\|_{H^1}^2 \|n_{+0}\|_{L^2}].
\end{eqnarray*}
Using (\ref{**}) and the definition of $\delta$ we arrive at
\begin{eqnarray*}
\lefteqn{ |H(Iu(\delta),n_+(\delta)) - H(Iu_0,n_{+0})| } \\
& \hspace{-0.2 cm} \le & \hspace{-0.2 cm} c_3 (N^{-1+} N^{3(1-s)}+
N^{-\frac{1}{2}+}N^{-(1-s)+} N^{3(1-s)} + (N^{-2+} + N^{-1+}
N^{-(1-s)+})N^{4(1-s)}).
\end{eqnarray*}
where $ c_3 = c_3(\overline{c},M) $. This is easily seen to be
bounded by $ \overline{c}N^{2(1-s)}$ (for large $N$).

The number of iteration steps to reach the given time $T$ is
$\frac{T}{\delta} \sim T N^{2(1-s)+} $. This means that in order to
give a uniform bound of the energy of the iterated solutions, namely
by $2\overline{c}N^{2(1-s)}$, from the last inequality the following
condition has to be fulfilled:
$$  c_3 TN^{2(1-s) +} (N^{-1+} N^{3(1-s)}+
N^{-\frac{1}{2}+}N^{-(1-s)+} N^{3(1-s)} + (N^{-2+} + N^{-1+}
N^{-(1-s)+})N^{4(1-s)}) < \overline{c} N^{2(1-s)} $$ where $c_3 =
c_3(2\overline{c},2M)$ (recall here that the initial energy is
bounded by $\overline{c}N^{2(1-s)}$).

One easily checks that this can be fulfilled by choosing $N\sim
T^{\frac{1}{2s-\frac{3}{2}}+}>>1$ provided $s > 3/4 $. So here is
the point where the decisive bound on $s$ appears.

A uniform bound of the energy implies by (\ref{2.3'}) uniform
control of
$$ \|\Lambda Iu(t)\|_{L^2} + \|n(t)\|_{L^2} + \|\Lambda^{-1}n_t(t)\|_{L^2} \le c
N^{1-s}.$$
Moreover $ \|Iu(t)\|_{L^2} \le \|u(t)\|_{L^2} = \|u_0\|_{L^2} $,
thus
$$ \|u(t)\|_{H^s} + \|n(t)\|_{L^2} + \|\Lambda^{-1} n_t(t)\|_{L^2} \le c
N^{1-s}.$$

This implies
$$ \sup_{0\le t \le T} (\|u(t)\|_{H^s} + \|n(t)\|_{L^2} +
\|\Lambda^{-1}n_t(t)\|_{L^2}) \le c
(1+T)^{\frac{1-s}{2s-\frac{3}{2}}+} .$$

\end{proof}

\begin{rm}

\end{rm}
\end{document}